\documentclass{amsart}
\usepackage{amsfonts,amssymb,amscd,amsmath,enumerate,verbatim,calc}
\usepackage{amscd,amssymb,amsopn,amsmath,amsthm,graphics,amsfonts,enumerate,verbatim,calc}
\usepackage[dvips]{graphicx}
\usepackage[colorlinks=true,linkcolor=red,citecolor=blue]{hyperref}
\input xy
\xyoption{all}
\calclayout

\newcommand{\rt}{\rightarrow}
\newcommand{\lrt}{\longrightarrow}

\newcommand{\st}{\stackrel}

\newcommand{\la}{\lambda}
\newcommand{\La}{\Lambda}

\newcommand{\fm}{\frak{m}}

\newcommand{\C}{\mathbb{C} }
\newcommand{\D}{\mathbb{D} }
\newcommand{\K}{\mathbb{K} }
\newcommand{\N}{\mathbb{N} }
\newcommand{\Z}{\mathbb{Z} }

\newcommand{\CA}{\mathcal{A} }

\newcommand{\CF}{\mathcal{F} }

\newcommand{\CQ}{\mathcal{Q} }

\newcommand{\CT}{\mathcal{T} }

\newcommand{\CV}{\mathcal{V}}
\newcommand{\CU}{\mathcal{U}}

\newcommand{\BZ}{\mathbf{Z}}

\newcommand{\CCF}{{\rm Cot}\mbox{-} \mathcal{F}}

\newcommand{\X}{\mathbf{X}}
\newcommand{\Y}{\mathbf{Y}}

\newcommand{\PP}{\mathbf{P}}

\newcommand{\Mod}{{\rm{Mod\mbox{-}}}}

\newcommand{\mmod}{{\rm{{mod\mbox{-}}}}}

\newcommand{\Inj}{{\rm{Inj}\mbox{-}}}
\newcommand{\Prj}{{\rm{Prj}\mbox{-}}}
\newcommand{\prj}{{\rm{prj}\mbox{-}}}

\newcommand{\inj}{{\rm{inj}\mbox{-}}}

\newcommand{\op}{{\rm{op}}}

\newcommand{\ac}{{\rm{ac}}}

\newcommand{\bb} {{\rm{b}}}

\newcommand{\gldim}{{\rm{gl.dim}}}

\newcommand{\Hom}{{\rm{Hom}}}
\newcommand{\Ext}{{\rm{Ext}}}
\newcommand{\End}{{\rm{End}}}

\theoremstyle{plain}
\newtheorem{theorem}{Theorem}[section]
\newtheorem{corollary}[theorem]{Corollary}
\newtheorem{lemma}[theorem]{Lemma}

\newtheorem{proposition}[theorem]{Proposition}

\theoremstyle{definition}

\newtheorem{remark}[theorem]{Remark}

\newtheorem{s}[theorem]{}

\theoremstyle{plain}

\theoremstyle{definition}

\numberwithin{equation}{section}

\begin{document}

\title[Duality and Serre functor]{Duality and Serre functor in homotopy categories }

\author[Asadollahi, Asadollahi, Hafezi, Vahed]{J. Asadollahi, N. Asadollahi, R. Hafezi and R. Vahed}

\address{Department of Mathematics, University of Isfahan, P.O.Box: 81746-73441, Isfahan, Iran}
\email{asadollahi@ipm.ir, asadollahi@sci.ui.ac.ir}

\address{Department of Mathematics, University of Isfahan, P.O.Box: 81746-73441, Isfahan, Iran}
\email{n.asadollahi@sci.ui.ac.ir}

\address{School of Mathematics, Institute for Research in Fundamental Sciences (IPM), P.O.Box: 19395-5746, Tehran, Iran }
\email{hafezi@ipm.ir}

\address{Department of Mathematics, Khansar Faculty of Mathematics and Computer Science, Khansar, Iran  }
\email{vahed@ipm.ir, vahed@khansar-cmc.ac.ir}




\begin{abstract}
For a (right and left) coherent ring $A$, we show that there exists a duality between homotopy categories $\K^{\bb}(\mmod A^{\op})$ and $\K^{\bb}(\mmod A)$. If $A=\La$ is an artin algebra of finite global dimension, this duality restricts to a duality between their subcategories of acyclic complexes, $\K^{\bb}_{\ac}(\mmod \La^{\op})$ and $\K^{\bb}_{\ac}(\mmod \La).$  As a result, it will be shown that, in this case, $\K_{\ac}^{\bb}(\mmod \La)$ admits a Serre functor and hence has Auslander-Reiten triangles.
\end{abstract}

\maketitle


\section{Introduction}
A contravariant functor between two categories that is an equivalence is called a duality. The role and importance of dualities is known in representation theory of algebras. Let $A$ be a right and left coherent ring. In this paper, we introduce and study a duality between the bounded homotopy categories of finitely generated right and finitely generated left $A$-modules, denoted by
$\K^{\bb}(\mmod A)$ and $\K^{\bb}(\mmod A^{\op})$, respectively. We gain this duality starting from an equivalence
$$ \mu: \D(\Mod (\mmod A^{\op})) \lrt \D(\Mod (\mmod A)^{\op})$$
of derived categories of functor categories.

The relationship between $\mu$ and some known dualities will be discussed. In particular, it is shown that, Proposition \ref{AusGJ-Dualtity} below,  there is a close relationship between $\mu$ and the Auslander-Gruson-Jensen duality
 $$\mathfrak{D} : \mmod (\mmod A^{\op})^{\op} \lrt \mmod (\mmod A)^{\op}.$$

Let $\La$ be an artin algebra of finite global dimension over a commutative artinian ring $R$. We show that in this case, the above  duality between $\K^{\bb}(\mmod \La^{\op})$ and $\K^{\bb}(\mmod \La)$ restricts to a duality between $\K^{\bb}_{\ac}(\mmod \La^{\op})$ and $\K^{\bb}_{\ac}(\mmod \La),$ where for an abelian category $\CA$, $\K_{\ac}^{\bb}(\CA)$ is the full subcategory of $\K^{\bb}(\CA)$ consisting of all acyclic complexes. This, in turn, implies that there is an equivalence of triangulated categories
\[\xymatrix{ \frac{\K^{\bb}(\mmod \La)}{\K^{\bb}(\prj \La)}  \ar[r]^{\sim} & \frac{\K^{\bb}(\mmod \La)}{\K^{\bb}(\inj \La)}.}\]
Note that under certain conditions, the quotient $\frac{\K^{\bb}(\mmod \La)}{\K^{\bb}(\prj \La)}$ is equivalent to the relative singularity category introduced and studied recently in \cite{KY}, see Remark \ref{KY}.

Finally, we show that $\K_{\ac}^{\bb}(\mmod \La)$ admits a Serre functor $\mathbb{S}$ in the sense of \cite{BK}. By a well-known result of Reiten and Van den Bergh \cite[Theorem I.2.4]{RV}, the existence of a Serre functor is equivalent to the existence of Auslander-Reiten triangles in a category and so we deduce that $\K_{\ac}^{\bb}(\mmod \La)$ admits Auslander-Reiten triangles.

\section{Preliminaries}\label{Section 2}
Throughout the paper, $A$ denotes a right and left coherent ring. $A$-module means right $A$-module. $\Mod A$, resp. $\mmod A$, denotes the category of $A$-modules, resp. finitely presented $A$-modules. $\Prj A$, resp. $\prj A$, denotes the full subcategory of $\Mod A$, resp. $\mmod A$, consisting of projective $A$-modules. $\Inj A$ and $\inj A$ represent injectives and finitely presented injectives, resp. For an additive category $\CA$, $\D(\CA)$, resp. $\K(\CA)$, denotes the derived category, resp. homotopy category, of $\CA$. As usual, the bounded derived, resp. homotopy, category of $\CA$, will be denoted by $\D^{\bb}(\CA)$, resp. $\K^{\bb}(\CA)$.

\s Following Auslander we let $\Mod(\mmod A)$, resp. $\Mod(\mmod A)^{\op}$, denote the category of all contravariant, resp. covariant, additive functors from $\mmod A$ to $\CA b$, the category of abelian groups. Throughout we shall use parenthesis to denote the Hom sets. An object $F$ of $\Mod(\mmod A)$, resp. $\Mod(\mmod A)^{\op}$, is called coherent if there exists a short exact sequence
\[ \ \ \ \ \ \ \ \  (-, X) \lrt (-,Y) \lrt F \lrt 0, \]
\[{\rm resp.}  \ \ (X, - ) \lrt (Y, - ) \lrt F \lrt 0,  \]
of functors, where $X$ and $Y$ belong to $\mmod A$. We let $\mmod(\mmod A)$, resp. $\mmod(\mmod A)^{\op}$, denote the full subcategory of  $\Mod(\mmod A)$, resp. $\Mod(\mmod A)^{\op}$, consisting of all coherent functors. It is known \cite{As1} that both $\Mod(\mmod A)$ and $\mmod(\mmod A)$ and also their counterparts of covariant functors are abelian categories with enough projective objects.

Special objects of such categories have been studied by several authors. In particular, an object $F$ of $\Mod(\mmod A)$ is flat if and only if there exists an $A$-module $M$ such that $F \cong \Hom_A(-, M)$, see \cite[Theorem B.10]{JL}. We let $\CF(\mmod A)$ denote the full subcategory of $\Mod(\mmod A)$ consisting of all flat functors.

\s \label{Krause-Stov} A sequence $0 \rt X \rt Y \rt Z \rt 0$ of $A$-modules is called pure exact if the induced sequence
\[0 \lrt X\otimes_AM \lrt Y\otimes_AM \lrt Z\otimes_AM \lrt 0,\]
is exact, for every left $A$-module $M$. A module $P$ is called pure projective, resp. pure injective, if it is projective, resp. injective, with respect to the class of all pure exact sequences. We let ${\rm PPrj} \mbox{-} A$, resp. ${\rm PInj} \mbox{-} A$, denote the full subcategory of $\Mod A$ consisting of all pure projective, resp. pure injective, $A$-modules.

The derived category of $A$ with respect to the pure exact structure is called pure derived category and is denoted by $\D_{\rm pur}(\Mod A)$. Krause \cite{K12} introduced and studied this category in deep. He \cite[Corollary 6]{K12} proved that for a ring $A$, there exists a triangle equivalence
\[\D_{\rm pur}(\Mod A) \simeq \D(\Mod (\mmod A)).\]

\s \label{Equi-1} Let $\K_{\ac}(\CF(\mmod A))$ be the full subcategory of $\K(\CF(\mmod A))$ formed by all acyclic complexes of flat functors. It is a thick subcategory of $\K(\CF(\mmod A))$ and so we have the quotient category $\frac{\K(\CF(\mmod A))}{\K_{\ac}(\CF(\mmod A))}$. By \cite[Lemma 4.4]{AAHV}, there exists a triangle equivalence
$$\varphi: \D(\Mod (\mmod A)) \lrt \frac{\K(\CF(\mmod A))}{\K_{\ac}(\CF(\mmod A))}.$$
This equivalence maps every complex $\X$ in $\D(\Mod (\mmod A))$ to a complex ${\bf F}$ in $\K(\CF(\mmod A))$, where ${\bf F}$ fits into a short exact sequence
$$ 0 \lrt {\bf C} \lrt {\bf F} \lrt \X \lrt 0$$
in $\C(\Mod (\mmod A))$ in which $\Ext^i({\bf F}' , {\bf C})=0$, for $i>0$ and for all ${\bf F}' \in \C(\CF(\mmod A))$.

\s \label{Equi-2} Recall that an object $C$ of an abelian category $\CA$ is called cotorsion if for every flat object $F$, $\Ext^1(F,C)=0$. Let ${\rm Cot\mbox{-}\CF}(\mmod A)$ denote the full subcategory of $\CF(\mmod A)$ consisting of all cotorsion-flat functors. By \cite[Theorem 4]{H}, a flat functor $(-,M)$ in $\Mod(\mmod A)$ is cotorsion if and only if $M$ is a pure-injective module. So the fully faithful functor $U: \Mod A \lrt \Mod (\mmod A)$ induces an equivalence
$$\K({\rm P}\Inj A) \st{\st{\K(U)}\sim} \lrt \K({\rm Cot}\mbox{-} \CF(\mmod A)),$$
of triangulated categories.

\s \label{Equi-3} By Remark 4.5 of \cite{AAHV}, for every complex ${\bf F}$ in $\K(\CF(\mmod A))$, there is a triangle
$$  {\bf G} \lrt {\bf F} \lrt {\bf C} \rightsquigarrow$$
in $\K(\CF(\mmod A))$  with ${\bf C} \in \K( {\rm Cot\mbox{-}\CF}(\mmod A))$ and ${\bf G} \in \K_{\ac}(\CF(\mmod A))$. So, there is a triangle functor
$$ \psi : \frac{\K(\CF(\mmod A))}{\K_{\rm ac}(\CF(\mmod A))} \lrt \K(\CCF(\mmod A)),$$
given by $\psi ({\bf F})={\bf C}$; see \cite[Proposition 2.6]{Mi} for more details.

\section{Dualities of homotopy categories}\label{Section 3}
Throughout the section, $A$ is a right and left coherent ring. Our aim is to show that there is a duality of triangulated categories
$$ \K^{\bb}(\mmod A^{\op}) \lrt \K^{\bb}(\mmod A),$$
that restricts to a duality between their full subcategories of all acyclic complexes
$$\K^{\bb}_{\ac}(\mmod A^{\op}) \lrt \K^{\bb}_{\ac}(\mmod A).$$

\s \label{Equi} In view of \ref{Equi-1}, \ref{Equi-2} and \ref{Equi-3}, we get an equivalence $\Psi: \D(\Mod(\mmod A)) \rt \K({\rm PInj} \mbox{-} A)$ of triangulated categories, given by the following composition
$$ \Psi: \D(\Mod (\mmod A)) \st{\varphi}\lrt \frac{\K(\CF(\mmod A))}{\K_{\ac}(\CF(\mmod A))} \st{\psi}\lrt \K(\CCF(\mmod A)) \st{\K(U)^{-1}} \lrt \K({\rm P}\Inj A),$$
of equivalences. Throughout we will use this equivalence.

\begin{lemma}\label{Prop-equi}
There is an equivalence
$$\mu: \D(\Mod(\mmod A^{\op})) \st{\sim} \lrt \D(\Mod(\mmod A)^{\op})$$
of triangulated categories, given by $$\mu(\X)= - \otimes_A \Psi(\X),$$ where $\Psi$ is the equivalence introduced in \ref{Equi}.
\end{lemma}

\begin{proof}
By \cite{GJ}, an object $E$ of $\Mod(\mmod A)^{\op}$ is injective if and only if $E \cong - \otimes_A M$, for some $M \in  {\rm PInj} \mbox{-} A^{\op}$. Therefore, the full and faithful functor ${\rm P}\Inj A^{\op} \lrt \Mod (\mmod A)^{\op}$, given by the rule $M \mapsto - \otimes_A M$, induces  the following triangle equivalence
$$ \K({\rm PInj}\mbox{-} A^{\op}) \simeq \K(\Inj(\Mod(\mmod A)^{\op})).$$
Moreover, similar to Lemma 4.8 of \cite{K}, one can see that the canonical functor $$\K(\Inj(\Mod(\mmod A)^{\op})) \lrt \D(\Mod(\mmod A)^{\op})$$ is an equivalence of triangulated categories. So, in view of the equivalence $$\Psi: \D(\Mod(\mmod A^{\op})) \lrt \K({\rm PInj} \mbox{-} A^{\op})$$ of \ref{Equi}, we get the desired equivalence  $$\mu : \D(\Mod (\mmod A^{\op})) \st{\sim} \lrt \D(\Mod (\mmod A)^{\op}).$$
\end{proof}

\begin{lemma}\label{rem2-def}
Let $\mu$ be the equivalence of Lemma \ref{Prop-equi}. Let $t: M \lrt N$ be an $A$-homomorphism of left $A$-modules. Then $\mu((-,M)) = - \otimes_A M $ and $\mu((-,t))=-\otimes_A t$.
\end{lemma}

\begin{proof}
Pick $M \in \Mod A^{\op}$ and consider the functor $(-, M)$ as a complex concentrated in degree zero, which is a complex of flat functors. Take an injective resolution
$$0 \lrt -\otimes_A M \st{-\otimes_A \varepsilon} \lrt {\bf I}_{-\otimes_AM}$$ of $- \otimes_A M$.
By the characterization of injective objects of $\Mod (\mmod A)^{\op}$, ${\bf I}_{-\otimes_A M}$ is of the form
$${\bf I}_{-\otimes_A M}: \ \ 0 \lrt -\otimes_A C^1 \st{-\otimes d^1} \lrt -\otimes_A C^2 \st{-\otimes d^2} \lrt -\otimes_A C^3 \lrt \cdots,$$
such that for each $i \in \N$, $C^i$ is pure-injective \cite{GJ}. Let ${\bf F}$ be  the complex
$$ 0 \lrt M \st{\varepsilon}\lrt C^1 \st{d^1} \lrt C^2 \st{d^2} \lrt \cdots$$
which is pure-exact and ${\bf C}$ be the following complex of pure-injectives
$$ {\bf C}: \ \ 0 \lrt  C^1 \st{ d^1} \lrt  C^2 \st{ d^2} \lrt  C^3 \lrt \cdots.$$
Hence, there exists a degree-wise split exact sequence
$$0 \rt (-,{\bf C}) \rt (-,{\bf F}) \rt (-, { M}) \rt 0$$
of complexes in $\Mod(\mmod A^{\op})$. So, there is the following triangle in $\K(\CF(\mmod A))$
$$(-,{\bf F}) \rt (-, { M}) \rt (-,{\bf C})[1] \rightsquigarrow $$
with $(-, {\bf F})\in \K_{\ac}(\CF(\mmod A))$ and $(-, {\bf C}) \in \K(\CCF(\mmod A))$.
Therefore, by definition, $\Psi((-,M)) = {\bf C}[1]$. Hence, $\mu((-,M))= -\otimes_A {\bf C} [1]$, which is quasi-isomorphic to $-\otimes_A M$.

Finally, natural transformation $(-,t): (-,M) \lrt (-,N)$ induces the natural transformation $-\otimes_A t: -\otimes_A M \lrt - \otimes_A N$. This, in turn, can be lifted to their injective resolutions. Hence, $\mu$ takes the morphism  $(-,t)$ to the morphism
$$-\otimes_A t: -\otimes_A M \lrt -\otimes_A N,$$
as it was claimed.
\end{proof}
\vspace{0.2cm}

Consider the functor $\mathfrak{D} : \mmod (\mmod A^{\op})^{\op} \lrt \mmod (\mmod A)^{\op}$ given by $$(\mathfrak{D}F)(N) = \Hom(F, N \otimes_A -),$$ where $F \in \mmod (\mmod A^{\op})^{\op}$ and $N \in \mmod A$. This is a duality first considered by Auslander \cite{A-Stable} and then, independently, proved by Gruson and Jensen \cite{GJ}. It is known as the Auslander-Gruson-Jensen duality.

In the following, we intent to show that there is a close relationship between the equivalence $\mu : \D(\Mod (\mmod A^{\op})) \lrt \D(\Mod (\mmod A)^{\op})$ in Lemma \ref{Prop-equi} and the Auslander-Gruson-Jensen duality $\mathfrak{D}$.

\begin{proposition}\label{AusGJ-Dualtity}
There exists a fully faithful contravariant functor $$\zeta: \mmod (\mmod A^{\op})^{\op} \lrt \D(\Mod (\mmod A^{\op}))$$
that commutes the following diagram
\[ \xymatrix{ \D(\Mod ( \mmod A^{\op} )) \ar[r]^{\st{\mu} \sim}& \D(\Mod (\mmod A)^{\op})
\\
\mmod (\mmod A^{\op})^{\op} \ar[u]^{\zeta} \ar[r]^{\mathfrak{D}} &\mmod (\mmod A)^{\op},\ar@{^(->}[u] }\]
where $\mathfrak{D}$ is the Auslander-Gruson-Jensen duality.
\end{proposition}

\begin{proof}
Let $F$ be an object of $\mmod (\mmod A^{\op})^{\op}$. By definition there is an exact sequence
$$ (\dag) \ \ \ \ 0 \lrt (M_2, -) \st{(d_2, -)} \lrt (M_1, -) \st{(d_1,-)} \lrt (M_0,-) \lrt F \lrt 0$$
with $M_i \in \mmod A^{\op}$, for $i \in \{0,1,2\}$.

We define $\zeta(F)$ to be the complex
\[\xymatrix@C=0.8cm@R=0.1cm{&&{\rm deg} 0 & {\rm deg} 1& {\rm deg} 2& &\\
\cdots \ar[r] & 0\ar[r] &(-,M_0) \ar[r]^{(-,d_1)}  & (-,M_1) \ar[r]^{(-,d_2)}  & (-,M_2) \ar[r] & 0 \ar[r] & \cdots.    }\]

Note that $\zeta(F)$ is a complex of projectives and one can easily check that $\zeta$ is a full and faithful functor.

Now, we compute the image of $\zeta(F)$ under the functor $\mu$.
First we show that $\mu$ maps the complex
\[\xymatrix@C=0.8cm@R=0.1cm{&& {\rm deg} 1& {\rm deg}2& &\\
\theta: \cdots \ar[r] & 0\ar[r]  & (-,M_1) \ar[r]^{(-,d_2)}  & (-,M_2) \ar[r] & 0 \ar[r] & \cdots    }\]
in $\D(\Mod (\mmod A^{\op}))$ to  the complex
\[\xymatrix@C=0.8cm@R=0.1cm{&&{\rm deg}1 & {\rm deg} 2&  &\\
\cdots \ar[r] & 0\ar[r]  & -\otimes_A M_1 \ar[r]^{-\otimes d_2}  & -\otimes M_2 \ar[r] & 0 \ar[r] & \cdots    }\]
in $\D(\Mod (\mmod A)^{\op})$. The complex $\theta$ is the mapping cone of the following morphisms of complexes
\[\xymatrix@C=0.8cm@R=0.1cm{&&  {\rm deg}2& &\\
\cdots \ar[r] & 0\ar[r]  & (-,M_1) \ar[ddd]^{(-,d_2)}  \ar[r] & 0 \ar[r] & \cdots  \\ \\ \\
 \cdots \ar[r] & 0\ar[r]  & (-,M_2)  \ar[r] & 0 \ar[r] & \cdots
 }\]
Since $\mu$ is a triangle functor, $\mu(\theta)$ is the mapping cone of the morphism
 \[\xymatrix@C=0.8cm@R=0.1cm{&&  {\rm deg}2& &\\
\cdots \ar[r] & 0\ar[r]  & \mu((-,M_1)) \ar[ddd]^{\mu((-,d_2))}  \ar[r] & 0 \ar[r] & \cdots  \\ \\ \\
 \cdots \ar[r] & 0\ar[r]  & \mu((-,M_2))  \ar[r] & 0 \ar[r] & \cdots.
 }\]
 By Lemma \ref{rem2-def}, the above diagram is isomorphic in $\D(\Mod (\mmod A)^{\op})$ to the diagram
 \[\xymatrix@C=0.8cm@R=0.1cm{&&  {\rm deg}2& &\\
\cdots \ar[r] & 0\ar[r]  & -\otimes_A M_1 \ar[ddd]^{-\otimes d_2}  \ar[r] & 0 \ar[r] & \cdots  \\ \\ \\
 \cdots \ar[r] & 0\ar[r]  & -\otimes_A M_2  \ar[r] & 0 \ar[r] & \cdots.
 }\]
 Thus, $\mu(\theta)$ is the complex
 \[\xymatrix@C=0.8cm@R=0.1cm{
\cdots \ar[r] & 0\ar[r]  & -\otimes_A M_1 \ar[r]^{-\otimes d_2}  & -\otimes_A M_2 \ar[r] & 0 \ar[r] & \cdots    }\]
with $-\otimes_A M_1$ at the $1$-th position.

Now, $\zeta(F)$ is the mapping cone of the following morphism of complexes
\[\xymatrix@C=0.8cm@R=0.1cm{&&  {\rm deg}1&& &\\
\cdots \ar[r] & 0\ar[r]  & (-, M_0) \ar[ddd]^{(-,d_1)}  \ar[r] & 0 \ar[r] & \cdots  &\\ \\ \\
 \cdots \ar[r] & 0\ar[r]  & (-, M_1)  \ar[r]^{(-,d_2)} & (-,M_2) \ar[r] & 0 \ar[r] &\cdots
 }\]
Hence, the same argument as above, applying this time to the above diagram, implies that $\mu(\zeta(F))$ is isomorphic in $\D(\Mod (\mmod A)^{\op})$ to the complex
\[\xymatrix@C=0.8cm@R=0.1cm{&& {\rm deg}0 &{\rm deg}1 & {\rm deg} 2&  &\\
\cdots \ar[r] & 0\ar[r]  & -\otimes_A M_0 \ar[r]^{-\otimes d_1} & -\otimes_A M_1 \ar[r]^{-\otimes d_2}  & -\otimes_A M_2 \ar[r] & 0 \ar[r] & \cdots.   }\]

Now, by applying $\mathfrak{D}$ to the exact sequence $(\dag)$, we have the following exact sequence
$$ 0 \lrt \mathfrak{D}F \lrt -\otimes_A M_0 \st{-\otimes d_1}\lrt - \otimes_A M_1 \st{-\otimes d_2}\lrt -\otimes_A M_2 \lrt 0.$$
Hence, if we consider $\mathfrak{D}F$ as a complex concentrated in degree zero, then $\mathfrak{D}F$ is quasi-isomorphic to $\mu(\zeta(F))$ in $\D(\Mod (\mmod A)^{\op})$. This completes the proof.
\end{proof}

\begin{remark}\label{ComGen}
It is known that for any ring $A$, the derived category $\D(\Mod A)$ is compactly generated. Moreover, the inclusion $\prj A \lrt \Mod A$ induces an equivalence between $\K^{\bb}(\prj A)$ and the full subcategory $\D(\Mod A)^{\rm c}$ of $\D(\Mod A)$ consisting of all compact objects \cite{Keller}. The same argument as in the ring case, implies that both $\D(\Mod (\mmod A^{\op}))$ and $\D(\Mod (\mmod A)^{\op})$ are compactly generated and
$$ \ \ \ \ \ \ \D(\Mod (\mmod A^{\op}))^{\rm c} \simeq \K^{\bb}(\prj (\Mod (\mmod A^{\op}))), \ \text{and}$$ $$\D(\Mod (\mmod A)^{\op})^{\rm c} \simeq \K^{\bb}(\prj (\Mod (\mmod A)^{\op})).$$
\end{remark}

\begin{theorem}\label{GrothDuality}
Let $A$ be a right and left coherent ring. There is the following duality of triangulated categories
$$\phi: \K^{\bb}(\mmod A^{\op}) \lrt \K^{\bb}(\mmod A).$$
\end{theorem}

\begin{proof}
Since $\mu$ preserves direct sums, it preserves compact objects. So, by Remark \ref{ComGen}, it induces the equivalence
$$ \mu| : \K^{\bb}(\prj (\Mod (\mmod A^{\op}))) \lrt \K^{\bb}(\prj (\Mod (\mmod A)^{\op}))$$
of triangulated categories. Moreover, Yoneda functors $u: \mmod A^{\op} \lrt \Mod (\mmod A^{\op})$ and $v: \mmod A \lrt \Mod (\mmod A)^{\op}$ yield the following equivalences of triangulated categories
$$\ \ \ \ \ \ \bar{u}: \K^{\bb}(\mmod A^{\op}) \lrt \K^{\bb}(\prj (\Mod (\mmod A^{\op}))) , \ \text{and}$$
$$\bar{v}: \K^{\bb}(\mmod A)^{\op} \lrt \K^{\bb}(\prj (\Mod (\mmod A)^{\op})).$$

Consequently, we have the following commutative diagram whose rows are triangle equivalences
\begin{equation}\label{5.5}
 \xymatrix@C=0.1cm@R=0.3cm{ & \D(\Mod ( \mmod A^{\op}))  \ar[rr]^{\mu}  \ar@{<-}[dd] && \D(\Mod (\mmod A)^{\op}) \ar@{<-}[dd]\\ \\
& \K^{\bb}(\prj (\Mod (\mmod A^{\op}))) \ar[rr]^{{\mu|}  } \ar@{<-}[dd]^{\bar{u}} && \K^{\bb}(\prj (\Mod (\mmod A)^{\op})) \ar@{<-}[dd]^{\bar{v}} \\ \\
& \K^{\bb}(\mmod A^{\op})  && \K^{\bb}(\mmod A)^{\op}
}
\end{equation}
Therefore we get a duality $\phi: \K^{\bb}(\mmod A^{\op})  \lrt \K^{\bb}(\mmod A)$, as desired.
\end{proof}

\begin{corollary}\label{Cor-equi}
There is the following duality of triangulated categories
$$\bar{\phi}: \frac{\K^{\bb}(\mmod A^{\op})}{\K^{\bb}(\prj A^{\op})} \lrt \frac{\K^{\bb}(\mmod A)}{\K^{\bb}(\prj A)}.$$
\end{corollary}

\begin{proof}
First we claim that the equivalence $\phi: \K^{\bb}(\mmod A^{\op}) \lrt \K^{\bb}(\mmod A)^{\op}$ can be restricted to the equivalence $ \phi| : \K^{\bb}(\prj A^{\op} )\lrt \K^{\bb}(\prj A)^{\op}$. Indeed, let $P$ be a finitely generated projective left $A$-module. Then by Lemma \ref{rem2-def}, $\mu((-,P)) = (- \otimes_A P)$ and $(-\otimes_A P) \cong (P^*,-) $. Hence, $(\phi|)(P)= P^*$ and so belongs to $\K^{\bb}(\prj A)$. So, using an induction argument on the length of the complexes of
$\K^{\bb}(\prj A^{\op})$ one can deduce that, the functor $\phi$ takes any bounded complex over $\prj A^{\op}$ to a bounded complex over $\prj A$. So, there is the following commutative diagram
\[\xymatrix{ \frac{\K^{\bb}(\mmod A^{\op})}{\K^{\bb}(\prj A^{\op})} \ar[r]^{\bar{\phi}} &  \frac{\K^{\bb}(\mmod A)}{\K^{\bb}(\prj A)} \\ \K^{\bb}(\mmod A^{\op}) \ar[r]^{\phi} \ar[u] & \K^{\bb}(\mmod A)\ar[u]\\
\K^{\bb}(\prj A^{\op} ) \ar[r]^{\phi|} \ar@{^(->}[u] & \K^{\bb}(\prj A), \ar@{^(->}[u]}\]
which implies the desired duality.
\end{proof}

\begin{remark}\label{KY}
Let $(S, \fm)$ be a commutative local complete Gorenstein $k$-algebra, where $k$ is an algebraically closed field. Let $M_0=S, M_1, ..., M_t$ be pairwise non-isomorphic indecomposable maximal Cohen-Macaulay $S$-modules. Set
$T :=\End_S(\bigoplus_{i=0}^tM_i).$ There is a fully faithful triangle functor $\K^\bb(\prj S) \lrt\D^\bb(\mmod T).$ By \cite[Definition 1.1]{KY}, the Verdier quotient
\[\frac{\D^\bb(\mmod T)}{\K^\bb(\prj S)}\]
is called the relative singularity category, denoted by $\Delta_S(T)$. This category recently has been studied in deep in \cite{KY}. We remark that, if $S$ is a self-injective algebra, then $\Delta_S(T)$ is equivalent to the quotient $\frac{\K^{\bb}(\mmod S)}{\K^{\bb}(\prj S)}.$ To see this, note that in this case all modules in $\mmod S$ are maximal Cohen-Macaulay, and so $T$ is the usual Auslander algebra of $S$, which is of finite global dimension, in fact, less than or equal to $2$.
\end{remark}

\section{Artin Algebras}\label{Artin Algebras}
In this section, we show that if $\La$ is an artin algebra of finite global dimension, then the duality introduced in Theorem \ref{GrothDuality}, restricts to a duality between their subcategories of acyclic complexes, $\K^{\bb}_{\ac}(\mmod \La^{\op})$ and $\K^{\bb}_{\ac}(\mmod \La).$

Throughout $\La$ is an artin $R$-algebra, where $R$ is a commutative artinian ring. We need some preparations.

\begin{lemma}\label{Lem-equi}
Let $\gamma: \K^{\bb}_{\ac}(\mmod \La) \lrt \frac{\K^{\bb}(\mmod \La)}{\K^{\bb}(\prj \La)} $ be the triangle functor taking every complex to itself. Then $\gamma$ is full and faithful. Furthermore, $\gamma$ is dense, and so is an equivalence, if and only if $\La$ has finite global dimension.
\end{lemma}

\begin{proof}
Let $f: \X \lrt \Y$ be a morphism in $\K^{\bb}_{\ac}(\mmod \La)$ such that $\gamma(f)$ vanishes in $\frac{\K^{\bb}(\mmod \La)}{\K^{\bb}(\prj \La)}$. So, there is a complex ${\bf Z} \in \K^{\bb}(\mmod \La)$ together with a morphism $s: {\bf Z} \lrt \X$ such that ${\rm cone}(s) \in \K^{\bb}(\prj \La)$ and $f \circ s$ is null homotopic. Hence, there is a morphism $t: {\rm cone}(s) \lrt \Y$ making the following diagram commutative
\[ \xymatrix{ {\bf Z} \ar[r]^s & \X \ar[r]^u \ar[d]^f & {\rm cone}(s) \ar@{.>}[dl]^t \ar@{~>}[r] &\\
& \Y &}\]
Since ${\rm cone}(s) \in \K^{\bb}(\prj \La)$ and $\Y$ is an acyclic complex, one can deduce that $t$, and so $f$, is null homotopic. Hence, $\gamma$ is faithful.
Assume that $\X \st{s}\leftarrow \BZ \st{f} \rt \Y$ is a roof in $\frac{\K^{\bb}(\mmod \La)}{\K^{\bb}(\prj \La)}$ with $\X$ and $\Y$ in $\K^{\bb}_{\ac}(\mmod \La)$. Consider a triangle ${\bf Z} \st{s} \rt \X \st{u} \rt {\rm cone}(s) \rightsquigarrow$ and then the composition
 $$ {\rm cone}(s)[-1] \st{u[-1]}\rt \BZ \st{f} \rt \Y.$$ Since  ${\rm cone}(s)\in \K^{\bb}(\prj \La)$ and $\Y$ is acyclic, the above morphism is null homotopic. So, there is a morphism $t: \X \lrt \Y$ making the following diagram commutative
\[ \xymatrix{ {\rm cone}(s)[-1] \ar[r] & \BZ \ar[r]^s \ar[d]^f & \X \ar@{~>}[r] \ar@{.>}[dl]^t &\\
& \Y & }\]
Now, the following commutative diagram implies that the roof $f \circ s^{-1}$ is equivalent  to a roof $t\circ {\rm id}^{-1}$ in  $\frac{\K^{\bb}(\mmod \La)}{\K^{\bb}(\prj \La)}$
 \[\xymatrix@C-0.8pc@R-0.5pc{ & & \BZ \ar[dl]_{\rm id} \ar[dr]^{s} & & \\
 & \BZ \ar[drrr]^{f}  \ar[dl]_s && \X \ar[dlll]_{\rm id} \ar[dr]^t &
 \\ \X & & & & \Y }\]
This means that $\gamma$ is full.

For the last part of the statement, let $\La$ has finite global dimension and $M$ be a finitely presented $\La$-module. Take a finite projective resolution $\PP_M \st{\pi} \lrt M$ of $M$. The mapping cone ${\rm cone}(\pi)$ belongs to $\K^{\bb}_{\ac}(\mmod \La)$ and is isomorphic to $M$ in $\frac{\K^{\bb}(\mmod \La)}{\K^{\bb}(\prj \La)}$. Therefore, $\gamma$ is dense and hence is an equivalence of triangulated categories.

For the converse, let $S$ be a simple $\La$-module. Since $\gamma$ is dense, there is a bounded acyclic complex $\X \in \K^{\bb}_{\ac}(\mmod \La)$ which is isomorphic to $S$ in $\frac{\K^{\bb}(\mmod \La)}{\K^{\bb}(\prj \La)}$. Hence, we have a roof $\xymatrix{ S & {\bf Z} \ar[r]^q \ar[l]_s & \X}$ such that ${\rm cone}(s)$ and ${\rm cone}(q)$ belong to $\K^{\bb}(\prj \La)$. As $\X$ is an acyclic complex, a triangle $ {\bf Z} \st{q}\rt \X \rt {\rm cone}(q) \rightsquigarrow$ implies that $\Z$ is quasi-isomorphic to ${\rm cone}(q)$.
Now, consider the image of a triangle ${\bf Z} \st{s} \rt S \rt {\rm cone}(s) \rightsquigarrow$ in $\D^{\bb}(\mmod \La)$. Since ${\bf Z}$ and ${\rm cone}(s)$ are isomorphic in $\D^{\bb}(\mmod \La)$ to bounded complexes of finitely generated projective $\La$-modules, so is $S$. This means that $S$ has finite projective dimension. The proof is now complete.
\end{proof}

The argument in the proof of the above lemma carry over verbatim to yield the following lemma.

\begin{lemma}\label{Lem-equi2}
Let $\gamma': \K^{\bb}_{\ac}(\mmod \La) \lrt \frac{\K^{\bb}(\mmod \La)}{\K^{\bb}(\inj \La)} $ be a triangle functor taking every complex to itself. Then $\gamma'$ is full and faithful. Furthermore, $\gamma'$ is dense, and so is an equivalence, if and only if $\La$ has finite global dimension.
\end{lemma}

\begin{proposition}
Let $\La$ be an artin algebra of finite global dimension. Then there is the following duality of triangulated categories
$$ \K^{\bb}_{\ac}(\mmod \La^{\op}) \simeq \K^{\bb}_{\ac}(\mmod \La).$$
\end{proposition}

\begin{proof}
Corollary \ref{Cor-equi} in conjunction with Lemma \ref{Lem-equi} imply the result.
\end{proof}

\begin{lemma}\label{Lem-embed}
Let $\La$ be an artin algebra of finite global dimension. Then there is a full and faithful functor
$$ \la: \underline{\rm mod} \mbox{-}\La \lrt \frac{\K^{\bb}(\mmod \La)}{\K^{\bb}(\prj \La)}$$
taking every $\La$-module $M$ to itself.
\end{lemma}

\begin{proof}
By Lemma \ref{Lem-equi}, we have the following equivalence of triangulated categories
$$\K^{\bb}_{\ac}(\mmod \La) \st{\gamma}\lrt \frac{\K^{\bb}(\mmod \La)}{\K^{\bb}(\prj \La)}.$$
So, it is enough to show that the functor
$$\underline{\rm mod} \mbox{-} \La \lrt \K_{\ac}^{\bb}(\mmod \La)$$
mapping every module $M$ in $\underline{\rm mod}\mbox{-} \La$ to the mapping cone of a projective resolution $\PP_M \lrt M$ of $M$, is full and faithful. This follows easily from the properties of projective resolutions.
Hence, the functor $\la$ which is the composition of the above two functors is also full and faithful.
\end{proof}

\begin{lemma}\label{Lem-embed2}
Let $\La$ be an artin algebra of finite global dimension. Then there is a full and faithful functor
$$\la': {\overline{{\rm mod}}}{\mbox{-}}\La \lrt \frac{\K^{\bb}(\mmod \La)}{\K^{\bb}(\inj \La)}$$
taking every $\La$-module $M$ to itself.
\end{lemma}

\begin{proof}
The proof is similar to the proof of the above lemma. Just one should consider the functor $\varsigma: {\overline{{\rm mod}}}{\mbox{-}}\La \lrt \K^{\bb}_{\ac}(\mmod \La)$ which maps every module $M$ to the mapping cone of an injective resolution $M \lrt {\bf E}_M$ of $M$, show that $\varsigma$ is full and faithful and then apply Lemma \ref{Lem-equi2}. We skip the details.
\end{proof}

Let $P_0 \st{f} \lrt P_1 \lrt M \lrt 0$ be a projective presentation of the $A$-module $M$. Recall that the Auslander transpose of $M$, ${\rm Tr} M$, is defined to be the $A^{\op}$-module ${\rm Coker}(\Hom_A(f,A))$. It is known that ${\rm Tr} M$ is unique up to projective equivalences and so induces an equivalence
$${\rm Tr}: {\underline{{\rm mod}}}{\mbox{-}} A^{\op} \lrt ({\underline{{\rm mod}}}{\mbox{-}} A)^{\op}$$
of stable categories.

\begin{proposition}\label{Comm-Tr}
Let $\La$ be an artin algebra of finite global dimension. Then there is the following commutative diagram
\[\xymatrix{ \frac{\K^{\bb}(\mmod \La^{\op})}{\K^{\bb}(\prj \La^{\op})}  \ar[r]^{\bar{\phi}[-2]} & \frac{\K^{\bb}(\mmod \La)}{\K^{\bb}(\prj \La)} \\
 {\underline{\rm mod}}{\mbox{-}}\La^{\op} \ar[u]_{\la} \ar[r]^{\rm Tr}   &  {\underline{\rm mod}}{\mbox{-}}\La \ar[u]^{\la},} \]
such that rows are dualities and $\la$ is the functor defined in Lemma \ref{Lem-embed}.
\end{proposition}

\begin{proof}
Let $M$ be a finitely presented left $\La$-module with no projective direct summands. Consider the functor $(-,M)$ in $\D(\Mod(\mmod \La^{\op}))$.  By Lemma \ref{rem2-def}, $\mu((-,M)) = - \otimes_{\La} M$. If $Q \lrt P \lrt M \lrt 0$ is a projective presentation of $M$, then the exact sequence
$$ 0 \lrt ({\rm Tr}M,-) \lrt (Q^*, -) \lrt (P^*, -) \lrt - \otimes_{\La} M\lrt 0$$
implies that $\mu((-,M))$ is isomorphic to the following complex in $\D(\Mod (\mmod \La)^{\op})$
$$0 \lrt ({\rm Tr}M, -) \lrt (Q^*,-) \lrt (P^*,-) \lrt 0.$$
It follows from the definition of $\phi$, see diagram (\ref{5.5}), that $\phi(M)$ is a complex
$$0 \lrt P^* \lrt Q^* \lrt {\rm Tr}M \lrt 0$$
in $\K^{\bb}(\mmod \La)$. This clearly is isomorphic to ${\rm Tr}M[2]$ in $\frac{\K^{\bb}(\mmod \La)}{\K^{\bb}(\prj \La)}$.
\end{proof}

\begin{remark}
It is known that if $\La$ is an artin $R$-algebra, there is the duality
$$ D: \mmod \La \lrt \mmod \La^{\op}$$
defined by $D(M)= \Hom_{R}(M, E)$, where $E$ is the injective envelope of $R/{\rm rad}R$. The duality $D: \mmod \La \lrt \mmod \La^{\op}$ can be extended to the following duality of triangulated categories
$$\K^{\bb}(\mmod \La) \st{\sim} \lrt \K^{\bb}(\mmod \La^{\op}).$$
Moreover, it is known that the duality $D$ can be restricted to the duality
$$D|: \prj \La \st{\sim}\lrt \inj \La^{\op}.$$
This duality also can be extended to the duality
$$\K^{\bb}(\prj \La) \st{\sim} \lrt \K^{\bb}(\inj \La^{\op})$$
of triangulated categories. Therefore, we can deduce that $D$ induces a duality $$\frac{\K^{\bb}(\mmod \La)}{\K^{\bb}(\prj \La)} \simeq \frac{\K^{\bb}(\mmod \La^{\op})}{\K^{\bb}(\inj \La^{\op})} $$ of triangulated categories, which we denote it again by $D$.
\end{remark}

\begin{corollary}\label{final}
Let $\La$ be an artin algebra of finite global dimension. Then there is the following commutative diagram
\[ \xymatrix{ \frac{\K^{\bb}(\mmod \La)}{\K^{\bb}(\prj \La)} \ar[r]^{ \bar{\phi}[-2]} & \frac{\K^{\bb}(\mmod \La^{\op})}{\K^{\bb}(\prj \La^{\op})} \ar[r]^D & \frac{\K^{\bb}(\mmod \La)}{\K^{\bb}(\inj \La)}\\
\underline{\mmod }\La \ar[u]^{\la} \ar[rr]^{D{\rm Tr}} && \overline{\mmod} \La \ar[u]_{\la'}, }\]
such that $\la$ and $\la'$ are fully faithful and $D{\rm Tr}$ and $D \bar{\phi}[-2]$ are equivalences.
\end{corollary}

\begin{proof}
By Lemmas \ref{Lem-embed} and \ref{Lem-embed2}, $\la$ and $\la'$ are fully faithful, respectively. The commutativity of the diagram follows from Proposition \ref{Comm-Tr}.
\end{proof}

\section{Serre functor for $\K^{\bb}_{\ac}(\mmod \La)$}\label{Section 4}
In this section, we prove that if $\La$ is an artin algebra of finite global dimension, then $\K^{\bb}_{\ac}(\mmod \La)$ has Serre duality and will investigate the relationship between this Serre duality and the equivalence
\[\xymatrix{ \frac{\K^{\bb}(\mmod \La)}{\K^{\bb}(\prj \La)}  \ar[r]^{D \bar{\phi}[-2]}   & \frac{\K^{\bb}(\mmod \La)}{\K^{\bb}(\inj \La)}}\]
of Corollary \ref{final}.

Recall that for a Hom-finite Krull-Schmidt $R$-linear triangulated category $\CT$, a Serre functor is an auto-equivalence $\mathbb{S} : \CT \lrt \CT$ such that the Serre duality formula holds, that is, we have bifunctorial isomorphisms
\[D \CT(X, Y ) \cong \CT(Y, \mathbb{S}X), \ \ {\rm for \ all} \ \  X,Y \in \CT,\]
where $D$ is the duality $\Hom_R( - , R)$.

To begin, observe that the duality
$$ D: \mmod \La \lrt \mmod \La^{\op}$$ induces the duality
$$D: \mmod (\mmod \La) \lrt \mmod (\mmod \La)^{\op}$$
which maps a functor $F$ to the functor $DF(M)=D(F(M))$ for every $M \in \mmod \La$. So, every injective object in $\mmod (\mmod \La)$ is of the form $D \Hom_{\La}(M,-)$, for some $M \in \mmod \La$. Define a functor
$$\CV: \Prj (\mmod (\mmod \La)) \lrt \Inj (\mmod (\mmod \La))$$
by $\CV(\Hom_{\La}(-,M))= D\Hom_{\La}(M,-)$. It is an equivalence of categories. $\CV$ can be naturally extended to the equivalence
$$\K^{\bb}(\Prj (\mmod (\mmod \La))) \st{\sim}\lrt \K^{\bb}(\Inj (\mmod (\mmod \La)))$$ of triangulated categories, which we denote it again by $\CV$.

Since global dimension of $\mmod(\mmod \La)$ is finite, $\CV$ is in fact the equivalence
$$\D^{\bb}(\mmod (\mmod \La)) \lrt \D^{\bb}(\mmod (\mmod \La)).$$
Now, the same argument as in the proof of Proposition 5.3 of \cite{ABHV} can be applied to prove that this functor is a Serre duality, i.e. for every $\X, \Y \in \D^{\bb}(\mmod (\mmod \La))$ there is the following natural isomorphism
$$\Hom(\X, \Y) \cong D \Hom(\Y, \CV \X),$$
where both Hom are taken in $\D^{\bb}(\mmod (\mmod \La))$.

\begin{proposition}
Let $\La$ be an artin algebra. Then $\K^{\bb}(\mmod \La)$ has Serre duality.
\end{proposition}

\begin{proof}
There is an equivalence $\D^{\bb}(\mmod (\mmod \La)) \simeq \K^{\bb}(\Prj (\mmod (\mmod \La)))$ of triangulated categories, because $\gldim(\mmod (\mmod \La))$ is finite. So, in view of Yoneda lemma, we get an equivalence
$$\CQ: \K^{\bb}(\mmod \La) \lrt \D^{\bb}(\mmod (\mmod \La))$$
of triangulated categories.
Let $\CU: \K^{\bb}(\mmod \La) \lrt \K^{\bb}(\mmod \La)$ be the functor that commutes the following diagram
\[ \xymatrix{ \D^{\bb}(\mmod (\mmod \La)) \ar[r]^{\CV} & \D^{\bb}(\mmod (\mmod \La)) \\
\K^{\bb}(\mmod \La) \ar[r]^{\CU} \ar[u]^{\CQ} & \K^{\bb}(\mmod \La) \ar[u]_{\CQ}.}\]
It can be easily checked that $\CU$ is also a Serre duality functor, i.e. for every complexes $\X$ and $\Y$ in $\K^{\bb}(\mmod \La)$ we have the following isomorphism
$$\Hom_{\K^{\bb}(\mmod \La)}(\X, \Y) \cong D \Hom_{\K^{\bb}(\mmod \La)}(\Y, \CU \X).$$
\end{proof}

\begin{remark}
The above proposition was proved by Backelin and Jaramillo \cite{BJ} using different approach. Their method is based on the construction of a $t$-structure in $\K^{\bb}(\mmod \La)$. It also can be obtained from \cite[Theorem 3.4]{ZH}. The proof presented here uses functor category techniques.
\end{remark}

\begin{proposition}\label{Kac-Serre}
Let $\La$ be an artin algebra of finite global dimension. Then $\K^{\bb}_{\ac}(\mmod \La)$  has Serre duality.
\end{proposition}

\begin{proof}
First note that the inclusion functor $i: \K^{\bb}_{\ac}(\mmod \La) \lrt \K^{\bb}(\mmod \La)$ admits a right adjoint $i_{\rho}: \K^{\bb}(\mmod \La) \lrt \K^{\bb}_{\ac}(\mmod \La)$. In fact, $i_{\rho}$ is defined as follows. Let $\X$ be a complex in $\K^{\bb}(\mmod \La)$. It has a K-injective resolution $\iota_{\X}: \X \lrt {\bf I}_{\X}$ with ${\bf I}_{\X} \in \K^{\bb}(\inj \La)$. Then $i_{\rho}(\X)={\rm cone}(\iota_{\X})[-1]$. We set $\mathbb{S}: \K_{\ac}^{\bb}(\mmod \La) \lrt \K^{\bb}_{\ac}(\mmod \La)$ to be the following composition of triangle functors
$$\K_{\ac}^{\bb}(\mmod \La) \st{i}\lrt \K^{\bb}(\mmod \La) \st{\CU}\lrt \K^{\bb}(\mmod \La) \st{i_{\rho}}\lrt \K^{\bb}_{\ac}(\mmod \La).$$
For every two complexes $\X$ and $\Y$ in $\K^{\bb}_{\ac}(\mmod \La)$, there are the following isomorphisms
\[ \begin{array}{lll}
\Hom_{\K^{\bb}_{\ac}(\mmod \La)}(\X, \Y) & \cong \Hom_{\K^{\bb}(\mmod \La)}(i\X, i\Y)\\
&\cong \Hom_{\K^{\bb}(\mmod \La)}(i\Y, \CU i \X) \\
& \cong \Hom_{\K^{\bb}_{\ac}(\mmod \La)}(\Y, i_{\rho}\CU i\X).
\end{array}\]
So $\mathbb{S}: \K^{\bb}_{\ac}(\mmod \La) \lrt \K^{\bb}_{\ac}(\mmod \La)$ is a Serre duality.
\end{proof}

Let $\CT$ be  a Hom-finite $R$-linear Krull-Schmidt triangulated category, where $R$ is a commutative artinian ring. A triangle $X \st{f} \lrt Y \st{g} \lrt Z \st{h} \lrt X[-1]$ in $\CT$ is an Auslander-Reiten triangle if the following conditions are satisfied
\begin{itemize}
\item [$(i)$] $X$ and $Z$ are indecomposable.
\item [$(ii)$] $h \neq 0$.
\item [$(iii)$] If $W$ is an indecomposable  object in $\CT$, then every non-isomorphism $t: W \lrt Z$ factors through $g$.
\end{itemize}

We say that $\CT$ has Auslander-Reiten triangles, if for every indecomposable object $W$ there exist Auslander-Reiten triangles starting and ending at $W$.

\begin{corollary}
Let $\La$ be an artin algebra of finite global dimension. Then the triangulated category $\K^{\bb}_{\ac}(\mmod \La)$ has Auslander-Reiten triangles.
\end{corollary}

\begin{proof}
It follows directly from Proposition \ref{Kac-Serre} and Theorem I.2.4 of \cite{RV}.
\end{proof}

The last theorem of the paper establishes a tight connection between the the functor $\mathbb{S}$ and the equivalence $D {\bar{\phi}[-2]}$  of Corollary \ref{final}.

\begin{theorem}
Let $\La$ be an artin algebra of finite global dimension. Then there is the following commutative diagram
\[ \xymatrix{ \frac{\K^{\bb}(\mmod \La)}{\K^{\bb}(\prj \La)} \ar[r]^{ \bar{\phi}[-2]} & \frac{\K^{\bb}(\mmod \La^{\op})}{\K^{\bb}(\prj \La^{\op})} \ar[r]^D & \frac{\K^{\bb}(\mmod \La)}{\K^{\bb}(\inj \La)}\\
\K^{\bb}_{\ac}(\mmod \La) \ar[u]^{\sim} \ar[rr]^{\mathbb{S}[-2]} && \K^{\bb}_{\ac}(\mmod \La) \ar[u]_{\sim}, }\]
where columns are equivalences of Lemmas \ref{Lem-equi} and \ref{Lem-equi2}.
\end{theorem}

\begin{proof}
Let $M$ be a finitely presented $\La$-module with no projective direct summands. Then $\CU(M)$ is isomorphic to the following complex
\[\xymatrix@C=0.8cm@R=0.1cm{&& {\rm deg}-2 &{\rm deg}-1 & {\rm deg}0&  &\\
\cdots \ar[r] & 0\ar[r]  & D{\rm Tr}M \ar[r] & DP_1^* \ar[r]^{Df^*}  & DP_0^* \ar[r] & 0 \ar[r] & \cdots,  }\]
where $P_1 \st{f} \lrt P_0 \lrt M\lrt 0$ is the minimal projective resolution of $M$.
Let
\[ \xymatrix{ \CU(M)\ar[d]^{\iota_{\CU(M)}} : & 0 \ar[r] & D{\rm Tr} M \ar[r] \ar[d] & D P_1^* \ar[r]^{Df^*} \ar[d] & DP_0^* \ar[r] \ar[d] & 0 \ar[r] \ar[d] & \cdots \\
{\bf I}: & 0 \ar[r] & I^0 \ar[r] & I^1 \ar[r] & I^2 \ar[r] & I^3 \ar[r] & \cdots \ar[r] & I^m \ar[r] & 0 }\]
be a K-injective resolution of $\CU(M)$. By definition $\mathbb{S}(M) = {\rm cone}(\iota_{\CU(M)})[-1]$ and so it is isomorphic in $\frac{\K^{\bb}(\mmod \La)}{\K^{\bb}(\inj \La)}$ to $D{\rm Tr}M[2]$. Hence, $\mathbb{S}[-2](M)=D{\rm Tr} M$.

On the other hand by Corollary \ref{final}, $D (\bar{\phi}[-2])(M)=D{\rm Tr}M$, for every $M \in \mmod A$.
Now, an induction argument on the length of the bounded complexes in $\K^{\bb}_{\ac}(\mmod \La)$ works to prove the commutativity of the desired diagram. See the proof of Proposition \ref{AusGJ-Dualtity} for similar argument.
\end{proof}


\end{document}